\documentclass[reqno,10pt]{amsart}

\usepackage{amsmath,amsthm,amssymb}

\theoremstyle{plain}
\newtheorem{theorem}{Theorem}[section]
\newtheorem{corollary}[theorem]{Corollary}
\newtheorem{proposition}[theorem]{Proposition}
\newtheorem{lemma}[theorem]{Lemma}

\theoremstyle{definition}
\newtheorem{definition}[theorem]{Definition}
\newtheorem{remark}[theorem]{Remark}

\newcommand{\R}{\mathbb{R}}
\newcommand{\C}{\mathbb{C}}

\newcommand{\N}{\mathbb{N}}
\newcommand{\e}{{\rm e}}

\newcommand{\norm}[1]{\left\lVert#1\right\rVert}
\newcommand{\abs}[1]{\left | {#1}\right | }
\newcommand{\all}[2]{ \{\, {#1} \, : \, {#2} \, \} }
\newcommand{\alll}[2]{ \left \{\, {#1} \, : \, {#2} \, \right \} }

\DeclareMathOperator{\Hdist}{Hdist}
\newcommand{\cp}{'}
\begin{document}

\title[Quantitative spectral perturbation theory for compact operators]
{Quantitative
  spectral perturbation theory for compact operators on a Hilbert space}
\author[A.~G\"{u}ven]{Ay\c{s}e G\"{u}ven}
\address{
Ay\c{s}e G\"{u}ven\\
Department of Mathematics\\
Faculty of Science and Arts\\
Ordu University\\
Ordu 52200\\
Turkey.
}
\email{ayseguvensarihan@odu.edu.tr}

\author[O.F.~Bandtlow]{Oscar F.~Bandtlow}
\address{
Oscar F.~Bandtlow\\
School of Mathematical Sciences\\
Queen Mary University of London\\
London E3 4NS\\
UK.
}
\email{o.bandtlow@qmul.ac.uk}



\date{May 26, 2020}

\begin{abstract}
We introduce compactness classes of Hilbert space operators by 
grouping together all operators for which the associated 
singular values decay at a certain speed and establish upper bounds for the norm of the resolvent of operators belonging to a particular compactness class.
As a consequence we obtain explicitly computable 
upper bounds for the Hausdorff distance of the spectra of two operators belonging to the same
compactness class
in terms of the
distance of the two operators in operator norm. 
\end{abstract}

\maketitle

\section{Introduction}

Perturbation theory is the study of the behaviour of characteristic 
data of a mathematical object when 
replacing it by a similar nearby object. 
More narrowly, spectral perturbation theory 
is concerned with the change of spectral data of linear operators 
(such as their spectrum, their eigenvalues and corresponding
eigenvectors) when the operators are subjected to a small perturbation. 

There are two sides to spectral perturbation theory, a qualitative
one and 
a quantitative one. 
Qualitative perturbation theory focusses on 
questions such as the continuity, differentiability and analyticity of eigenvalues and 
eigenvectors, 
while quantitative perturbation theory attempts to provide 
computationally accessible bounds for the smallness of the change in the 
spectral data in terms of the smallness of the perturbation. 

The book by Kato \cite{Kat76} is the main reference for spectral perturbation theory, focussing mostly on the qualitative part of the theory.  
Qualitative and quantitative aspects are discussed in the 
 article and book by Chatelin \cite{Cha81, Cha83} and the book by 
Hinrichsen and Pritchard \cite{HP11}.

The present article, located at the interface of functional 
analysis and linear algebra, addresses
the following problem of fundamental importance 
in both qualitative and quantitative 
perturbation theory.  If $A$ and $B$ are two compact operators acting on a separable Hilbert
space which are close, then how close are 
their spectra $\sigma(A)$ and $\sigma(B)$? 

In order to make this question more precise we need to specify metrics 
to measure distances of operators and spectra. Distances of operators
will typically be given by the underlying operator norm $\|\cdot \|$, while
distances of spectra 
will be determined by the Hausdorff metric (see below). 

A standard result in qualitative perturbation theory tells us that if $A$ and $B$ are 
compact operators and 
$\|A-B\|$ becomes vanishingly small, then so 
does the Hausdorff distance of their spectra
 (see, for example, \cite[Theorem~3]{New51}). 
However, this result does not give any quantitative information 
 on how large the Hausdorff distance of $\sigma(A)$ and $\sigma(B)$ is
 when $\|A-B\|$ is small but non-zero. 
 
Quantitative information of this type is interesting in situations
where 
one wants to determine the
spectrum of an arbitrary compact operator $A$ on a separable 
Hilbert space by numerical means. The standard approach to solving
this 
infinite-dimensional problem is to reduce it to a finite-dimensional one. This can, for example, be 
achieved as follows. 
Fix an orthonormal basis $(e_n)_{n\in \N}$ for the Hilbert space 
and define 
orthogonal projections onto the space spanned by the first $k$ basis vectors by setting 
\[ P_kx=\sum_{n=1}^{k}(x, e_n)e_n\,,\] 
where $(\cdot, \cdot)$ denotes the inner product of $H$. Now 
 \[ A_k=P_kAP_k \]
is a finite rank operator, the spectrum of which is in principle
computable, at least to arbitrary precision, since it boils down to
the computation of the 
eigenvalues of a matrix. 
Moreover, it is possible to show that this sequence of finite 
rank operators $(A_k)_{k\in \N}$
converges to $A$ in operator norm (see, for example, \cite[Theorem~4.1]{ALL01}). 
Thus, if quantitative bounds for the Hausdorff distance of the 
spectra of two compact operators are available, then the spectrum 
of $A$ can, in principle, be computed to arbitrary precision. In passing we note that 
the problem of determining 
the spectrum of an arbitrary bounded operator to a given precision is much more 
complicated (see \cite{Hansen}). 

In order to formulate the results of this article we require some notation. 
For $z\in \C$ and a compact subset $\sigma\subset \C$ let  
\[ d(z,\sigma)=\inf_{\lambda\in \sigma}\abs{z-\lambda} \] 
denote the distance of $z$ to $\sigma$. The \emph{Hausdorff distance} $\Hdist(\cdot, \cdot)$, also known as 
the \emph{Pompeiu-Hausdorff distance} 
(see \cite{BP13} for some historical background), 
is the following metric defined on the set of compact
subsets of $\C$
\[
\Hdist(\sigma_1, \sigma_2) 
 = \max\{ \hat d(\sigma_1, \sigma_2), \hat d(\sigma_2, \sigma_1)\}
\]
where 
\[
\hat d(\sigma_1, \sigma_2)= \sup_{\lambda \in\sigma_1} d(\lambda, \sigma_2)\,, 
\] 
and $\sigma_1$ and $\sigma_2$ are two compact
subsets of $\C$. It is easy to see that the Hausdorff distance is 
a metric on the set of compact subsets of $\C$.

Now recall the following notions from matrix perturbation theory: 
for two bounded operators $A$ and $B$, the \textit{spectral variation
of $A$ with respect to $B$} is defined to be 
\[ \hat{d}(\sigma(A), \sigma(B))\,,\]
while the \textit{spectral distance of $A$ and $B$} is 
\[ \Hdist(\sigma(A),\sigma(B)) \]
(see, for example, \cite[Chapter~8, Definition~8.4.1]{Gil03}).

The main concern of the present article is to provide 
explicit upper bounds for the Hausdorff distance 
of the spectra of two arbitrary compact operators $A$ and $B$ 
on a separable Hilbert space in terms of the distance of the two 
operators $A$ and $B$ in operator norm. This will be achieved by grouping together 
all compact operators for which the associated 
singular values decay at a certain speed into a class, 
termed a \emph{compactness class} (see  Definition~\ref{w-compact}).
These classes of operators generalise the exponential 
classes introduced by the second author (see \cite{Ban08}).

Our approach relies on an adaptation of finite-dimensional arguments going back to 
work of Henrici (see \cite{Hen62}), who obtained upper bounds for the spectral variation 
of two matrices as follows. In a first step, an upper bound
for the norm of the resolvent $(zI-A)^{-1}$ of a matrix $A$ is obtained which only 
depends on the distance of $z$ to the spectrum of $A$. This is achieved by writing the 
matrix $A$ as a perturbation of a normal matrix $D$ having the same spectrum as $A$ 
by a nilpotent matrix $N$. Explicit upper bounds for the spectral distance of two matrices 
can then be obtained in a second step, by using an argument going back to Bauer and 
Fike \cite{BF60}, which converts resolvent bounds into spectral distance bounds 
(see Theorem \ref{otbf}).

So far, infinite-dimensional analogues of these bounds have been obtained only for 
certain subclasses of compact operators. To the best of our knowledge, the first results 
in this direction are due to 
Gil\cp, who, in a series of 
papers begun in 1979, obtained spectral variation and distance bounds mostly for 
operators in the Schatten 
classes (see \cite{Gil95, Gil03} and references therein) and more recently for 
operators with inverses in the Schatten classes (see \cite{Gil12, Gil14}).  
Pokrzywa \cite{Pok85} has found similar bounds for operators in symmetrically normed 
ideals, while the second author obtained bounds, simpler and sharper than those of 
Gil\cp\ and Pokrzywa,  
for Schatten class operators \cite{Ban04} and for operators in exponential classes 
\cite{Ban08}. All three authors 
essentially use Henrici's approach to obtain their bounds, by first deriving resolvent 
bounds for quasi-nilpotent operators and 
then using the perturbation argument outlined above. For a completely different 
approach to obtain spectral variation bounds using determinants, see \cite{BG15}. 

This article is organised as follows. 
In Section~\ref{comfirst} we give the precise definition of compactness classes 
determined by 
the speed of decay of the singular values of the operators in the class
and study their functional analytic properties in some detail. In particular we shall find 
sufficient conditions guaranteeing that these classes of operators form quasi-Banach 
operator ideals in the 
sense of Pietsch (see \cite{Pie80, Pie86}).
In Section~\ref{comsecond} we shall use a theorem of Dostani\'{c} \cite{Dos01} to 
produce bounds for the resolvents of 
quasi-nilpotent operators in a given compactness class. Using the technique of Henrici 
discussed 
earlier we then obtain an upper bound for $\norm{(zI-A)^{-1}}$ for an arbitrary operator 
$A$ 
in a given compactness class, which depends only on the asymptotics of the singular 
values of $A$ and the 
distance of $z$ to the spectrum of $A$ (see Theorem~\ref{t1}).
The following Section~\ref{sec:particular} is devoted to studying the
behaviour of the bound for the norm of resolvents derived in the
previous section for two particular families of compactness classes
already in the literature. 
In Section~\ref{comthird}, the general resolvent bounds
obtained in Section~\ref{comsecond} together 
with 
the Bauer-Fike argument will yield the main result
of 
this article, an explicit
upper bound for the spectral distance of two operators in a given 
compactness class, depending only on 
the distance in operator norm of the operators and their 
respective departures from normality (see 
Theorem~\ref{t3}).
To the best of our knowledge, no bound for the spectral distance applicable to arbitrary 
compact operators has appeared in the literature yet. 
 A particular feature of this result is that it
turns 
out to be sharp for normal operators (see 
Remark~\ref{remark:sharpcompactnessbound} (iii)). 
In the final section we will briefly discuss an application of the
main 
result giving circular inclusion regions for pseudospectra of an 
operator in a given compactness class (see Theorem~\ref{th:pseudo}).

\section{Preliminaries}
In this section we fix notation and briefly recapitulate some
facts about compact operators on a Hilbert space which we rely on in the following. 

Let $H_1$ and $H_2$ be separable Hilbert spaces. 
We write $L(H_1, H_2)$ to denote the Banach space of bounded linear operators
from $H_1$ to $H_2$ equipped with the operator norm $\norm{\cdot}$ and 
$S_{\infty}(H_1, H_2)\subset L(H_1, H_2)$ to denote the closed subspace of
compact operators from $H_1$ to $H_2$. If $H=H_1=H_2$ we use the
short-hands $L(H)$ and $S_\infty(H)$ for $L(H_1,H_2)$ and 
$S_\infty(H_1, H_2)$, respectively. 

For $A\in L(H)$  the spectrum and the resolvent
set of $A$ will be denoted
by $\sigma(A)$ and $\rho(A)$, respectively. Moreover,  for 
$z\in \rho(A)$, we write 
$R(A; z)=(zI-A)^{-1}$ for the resolvent of $A$. 

For $A\in S_{\infty}(H)$ we use 
 $\lambda(A)=(\lambda_k(A))_{k\in \N}$ to denote its
eigenvalue sequence, counting algebraic multiplicities and ordered by
decreasing modulus so that 
\[ \abs{\lambda_1(A)}\geq \abs{\lambda_2(A)} \geq \cdots \]
If $A$ has only finitely many non-zero eigenvalues, we set
$\lambda_k(A)=0$ for $k>N$, where $N$ denotes the number of non-zero
eigenvalues of $A$. The symbol $\abs{\lambda(A)}$ 
will denote the sequence $(\abs{\lambda_k(A)})_{k\in \N}$.

Let now $A\in S_{\infty}(H_1,H_2)$. For $k\in \N$, the $k$-th singular value of $A$ is given by 
\[ s_k(A)=\sqrt{\lambda_n(A^*A)} \quad (k \in \N)\,, \]
where $A^*$ denotes the Hilbert space adjoint of $A$. 
For later use, we note that the singular values enjoy the following two properties. 
Given $A\in S_{\infty}(H_3, H_2)$, $B\in L(H_1, H_2)$ and $C\in L(H_4, H_3)$ we have 
\begin{equation}
\label{peqo}
 s_k(BAC)\leq \norm{B}s_k(A) \norm{C} \, \quad (\forall k\in \N) \,,
\end{equation}
while for $A, B \in S_{\infty}(H_1, H_2)$ we have 
  \begin{equation}
\label{ceqs}
 s_{k+l-1}(A+B) \leq s_k(A)+ s_l(B)\, \quad (\forall k,l\in \N)\,.
\end{equation}
Eigenvalues and singular values satisfy a number of inequalities known as Weyl's
inequalities. We give the most important one, known as the multiplicative Weyl inequality
(see \cite [Chapter~VI, Theorem~2.1]{GGK90}).

Let $A \in S_{\infty}(H)$. Then we have
\begin{equation}
\label{twyl1}
\prod_{k=1}^{n}|\lambda_k(A)| \leq \prod_{k=1}^{n}s_k(A)  \quad
(\forall n \in \N) \,.
\end{equation}

For more information about these notions see, for example, \cite{DS63,GK69,Pie86}.

\section{Compactness Classes}
\label{comfirst}
The basic idea to define these classes 
is to group together all compact operators on a separable Hilbert space 
the singular values of which 
decay at a certain speed, quantified by a 
given `weight sequence' (see Definition~\ref{w-compact}). 

The aim of this section is to examine the behaviour of compactness classes 
under addition and multiplication, to show that these classes are 
quasi-Banach operator ideals under suitable conditions on the weight sequence 
and to determine the decay rate of the eigenvalue sequence of
 an operator in a given compactness class. 

We start by defining the notion of a weight sequence. 

\begin{definition} Let 
\[\mathcal W= \all{w: \N\rightarrow \R^+_0}
{\text{$w_k\geq w_{k+1}, \,\forall k \in \N$ and $\lim_{k\to \infty}w_k=0$}}\,.\]
Elements of $\mathcal W$ will be referred to as \emph{weight sequences}, or simply \emph{weights}. 
\end{definition}

Every $w\in \mathcal W$ now gives rise to a compactness class as follows. 

\begin{definition}
\label{w-compact}
Let $w\in\mathcal W$. An operator $A\in S_\infty(H_1,H_2)$ is said to be \textit{$w$-compact} if there is a constant 
$M\geq 0$ such that 
 \begin{equation}
  \label{eq:wcompact}
     s_k(A) \leq M w_k \quad(\forall k\in \N)\,.
  \end{equation}
The infimum over all $M$ such that (\ref{eq:wcompact}) holds 
will be referred to as 
the \textit{$w$-gauge of $A$}
and will be denoted by $|A|_w$. 

The collection of all $w$-compact operators $A\in S_\infty(H_1,H_2)$ will be denoted by $E_w(H_1,H_2)$ or simply by $E_w(H)$ in case $H=H_1=H_2$. 
\end{definition}

For later use we also define the following sequence space analogues of compactness classes. 

\begin{definition}
Given $w\in\mathcal W$, let $\mathcal{E}_w$ denote
the set of all complex-valued sequences $(x_n)_{n\in\N}$ for which there is 
 a constant 
$M\geq 0$ such that 
 \begin{equation}
  \label{eq:wcompact2}
      \abs{x_k}\leq M w_k \quad(\forall k\in \N)\,.
  \end{equation}
The infimum over all $M$ such that (\ref{eq:wcompact2}) holds 
will be referred to as 
the \textit{$w$-gauge of $x$}
and will be denoted by $\abs{x}_w$.
\end{definition}

\begin{remark}
It is not difficult to see that $\mathcal{E}_w$ is a 
Banach space when equipped with the
$w$-gauge $\abs{\cdot}_w$. The situation is different for $E_w$, which need not even be 
a linear space in general
(see Proposition~\ref{p2}). 
\end{remark}
Compactness classes generalise classes that have already appeared in
the literature, such as the 
Schatten-Lorentz ideals $S_{p,\infty}$ (see, for example,
\cite[p.~481]{Pel85}), which correspond to the weights $w_k=k^{-1/p}$
with $p\in (0,\infty)$
or the `exponential classes' studied by Bandtlow (see \cite{Ban08}), 
which correspond to weights of the form 
$w_k=\exp(-ak^\alpha)$ with $a\in (0,\infty)$ and $\alpha \in (0,\infty)$.

We shall now explore some of the properties of $E_w(H_1, H_2)$ 
for a general weight $w$.
We start with the following elementary observation. 

\begin{proposition}\label{p0}
Let $v,w \in \mathcal W$. 
If there exists $M\geq 0$ such that $v_k \leq M w_k$ for every $k \in \N$ and $A \in E_v(H_1, H_2)$, then $A \in E_w(H_1, H_2)$ 
and ${\abs{A}}_w\leq M{\abs{A}}_v\,$. 
\end{proposition}
\begin{proof}
Suppose $A \in E_v(H_1, H_2)$ and there exists $M\geq 0$ such that $v_k\leq M w_k$ for every $k \in \N$. 
Then we have, for every $k \in \N$, 
 \[s_k(A) \leq {\abs{A}}_vv_k\leq {\abs{A}}_v M w_k \,.\]
 Hence we obtain $A \in E_w(H_1, H_2)$ and ${\abs{A}}_w\leq M{\abs{A}}_v$.
\end{proof}

The observation above motivates defining a partial 
order on $\mathcal W$ as follows 
\[v\preceq w :\iff \exists M \geq 0\quad \text{such that} \quad v_k\leq M w_k \quad (\forall k \in \N)\, .\]
We shall also define an equivalence relation on $\mathcal W$ by setting 
\[v\asymp w :\iff v\preceq w\quad \text{and} \quad w\preceq v \, .\]
Using the above partial order
we obtain the following inclusion. 

\begin{proposition}\label{p1}
Let $\dim H_1=\dim H_2 = \infty$ and let $v,w \in \mathcal W$. Then
\[v\preceq w \iff E_v(H_1, H_2)\subseteq E_w(H_1, H_2)\, .\]
\end{proposition}
\begin{proof}
For the forward implication we need to show that  
if $v\preceq w$ then $E_v(H_1, H_2) \subseteq E_w(H_1, H_2).$
This, however, follows directly from Proposition~\ref{p0}.

For the converse, suppose that $E_v(H_1, H_2)\subseteq E_w(H_1, H_2)$. 
We need to show that $v\preceq w$. 
Fix orthonormal bases $(e_k)_{k \in \N}$ for $H_1$ 
and $(f_k)_{k \in \N}$ for $H_2\,$.  
Define an operator $A\in L(H_1, H_2)$ by setting  $Ae_k=v_kf_k$ for every $k \in \N$. 
We clearly have $s_k(A)=v_k$ for every $k\in \N$, so $A \in E_v(H_1, H_2)$. But since 
$E_v(H_1, H_2)\subseteq E_w(H_1, H_2)$, 
we have $A \in E_w(H_1, H_2)$. Thus there exists $M\geq0$ such that $v_k=s_k(A)\leq M w_k\,$ for every $k\in \N$, so $v\preceq w$ and the backwards implication is proved as well. 
\end{proof}
\begin{corollary}\label{c1}
Let $\dim H_1=\dim H_2 = \infty$ and let $v,w \in \mathcal W$. Then
\[v\asymp w\iff E_v(H_1, H_2) = E_w(H_1, H_2)\,.\]
\end{corollary}
Although $E_w(H_1, H_2)$ is not a linear space in general, it is closed under 
multiplication by scalars and operators, as we shall see presently.  
\begin{lemma} \label{al}
If $A \in E_w(H_1, H_2)$ and $\alpha \in \C$, then 
\[
\alpha A \in E_w(H_1, H_2)\quad \text{and}\quad  {\abs{\alpha A}}_w=\abs{\alpha}{\abs{A}}_w\,.
\]
\end{lemma}
\begin{proof}
Let $A \in E_w(H_1, H_2)$ and $\alpha \in \C$. Then 
\[s_k(\alpha A)= |\alpha| s_k(A)\leq |\alpha|{|A|}_w w_k \quad (\forall k\in \N) \,,\] 
so $\alpha A \in E_w(H_1, H_2)$ and
\begin{equation}\label{ver}
{|\alpha A|}_w\leq |\alpha|{|A|}_w\,.
\end{equation}
It remains to prove that ${|\alpha A|}_w\geq |\alpha|{|A|}_w$ for every $\alpha \in \C$.
If $\alpha=0$, then there is nothing to prove. If $\alpha \not =0$, then using (\ref{ver}) we have
\[{|A|}_w={|{\alpha}^{-1}\alpha A|}_w \leq {|{\alpha}^{-1}|}{|\alpha A|}_w \, .\]Therefore
 we obtain, for every $\alpha \in \C$, 
\[|\alpha|{|A|}_w \leq {|\alpha A|}_w\,.\]
\end{proof}

\begin{proposition}\label{ideal}
If $B\in L(H_2, H_1)$, $A \in E_w(H_3, H_2)$ and $C \in L(H_4, H_3)$, then
 ${|BAC|}_w \leq \norm{B}{|A|}_w\norm{C}$. And
\[L(H_2, H_1)E_w(H_3, H_2)L(H_4, H_3)\subseteq E_w(H_4, H_1)\, .\]
\end{proposition}
\begin{proof}
Let $A\in E_w(H_3, H_2)$. By (\ref{peqo}), we obtain
\[s_k(BAC)\leq \norm{B}s_k(A)\norm{C}\leq \norm{B}{\abs{A}}_w\norm{C} w_k\quad (\forall k\in \N)\,.\]
Thus we have 
$BAC \in E_w(H_4, H_1)$ and ${\abs{BAC}}_w \leq \norm{B}{\abs{A}}_w\norm{C}$.
\end{proof}
\begin{remark}
\label{pre-ideal}
Note that Proposition~\ref{ideal} implies that 
\[L(H)E_w(H)L(H)\subseteq E_w(H)\, .\]
Hence $E_w(H)$ satisfies the second condition of 
the definition of an operator ideal (see, for example, \cite[1.1.1]{Pie80})
though not necessarily the first one, concerned with linearity. 
Thus $E_w(H)$ is what is sometimes referred to as a pre-ideal (see, for example, \cite{Nel82}). 
\end{remark}

We shall now investigate the behaviour of 
compactness classes under addition (see Proposition~\ref{p2}). Before doing so we require 
the following definition. 
\begin{definition}
Let $w\in \mathcal W$. Then $\dot w$ is the sequence obtained from $w$ by doubling each entry, that is, 
$\dot{w}=(w_1, w_1, w_2, w_2, w_3, w_3, \ldots)$. 
More precisely, $\dot{w}$ is the sequence given by 
\[
{\dot w}_k = \begin{cases} w_{\frac{k}{2}} &\mbox{if } k \quad\text{is even} \\ 
w_{\frac{k+1}{2}}  & \mbox{if } k \quad\text{is odd}. \end{cases}
\] 
\end{definition}

We are now ready to investigate how compactness classes behave under addition. 
\begin{proposition}\label{p2}
Let $w\in \mathcal{W}$.  
Then the following assertions hold. 
\begin{itemize}
\item[(i)] If $A,B\in E_w(H_1,H_2)$, then $A+B\in E_{\dot{w}}(H_1, H_2)$
with 
\[ |A+B|_{\dot{w}}\leq |A|_w+|B|_w\,.\]
\item[(ii)] If $\dim H_1=\dim H_2=\infty$, then assertion (i) is sharp 
in the sense that if there is $v\in \mathcal{W}$ such that 
 $A+B \in E_v(H_1, H_2)$ for all $A,B \in E_w(H_1, H_2)$, then $\dot w \preceq v$.
\end{itemize}
\end{proposition}
\begin{proof} \noindent
\begin{itemize}
\item[(i)]
Suppose $A,B \in E_w(H_1, H_2)$. 
Using (\ref{ceqs}) we have 
\[s_{2k-1}(A+B)\leq s_k(A)+s_k(B)\leq ({\abs{A}}_w +{\abs{B}}_w ) w_k=({\abs{A}}_w +{\abs{B}}_w )\dot{w}_{2k-1}\]
since $\dot{w}_{2k-1}=w_k$ for every $k\in \N$.
As the singular values are monotonically decreasing and $\dot{w}_{2k}=w_k$ for every $k \in \N$, 
we obtain
\[s_{2k}(A+B)\leq s_{2k-1}(A+B)\leq ({\abs{A}}_w +{\abs{B}}_w )w_k=({\abs{A}}_w +{\abs{B}}_w )\dot{w}_{2k}\,.\]
Hence we have
\[s_{k}(A+B)\leq ({\abs{A}}_w +{\abs{B}}_w ){\dot{w}}_k \quad(\forall k\in \N)\,.\]
Therefore
\[A+B \in E_{\dot{w}}\quad \text{and}\quad {\abs{A+B}}_{\dot{w}}\leq {\abs{A}}_w+ {\abs{B}}_w \,.\]
\item[(ii)]
Since both $H_1$ and $H_2$ are infinite-dimensional we can choose 
orthonormal bases $(e_k)_{k \in \N}$ for $H_1$ 
and $(f_k)_{k \in \N}$ for $H_2\,$.  
Define an operator $A\in L(H_1, H_2)$ by setting 
\[ 
Ae_k= \begin{cases} 0 & \text{if $k$ is even,} \\ 
w_{\frac{k+1}{2}}f_k  & \text{if $k$ is odd,} \end{cases} 
\]
and an operator $B\in L(H_1, H_2)$ by setting
\[ 
Be_k= \begin{cases} w_{\frac{k}{2}} f_k &\text{if $k$ is even,} \\ 
0  & \text{if $k$ is odd,} \end{cases}  
\]
Clearly, we have 
\[s_k(A)=s_k(B)=w_k \quad (\forall k\in\N) \,,\] so
 $A, B\in S_{\infty}(H_1, H_2)$.
At the same time we have 
\[s_k(A+B)={\dot w}_k \quad (k\in\N)\,,\]  
so $A+B \in E_v(H_1, H_2)$. 
Using the observation above, there exists $M\geq 0$ 
such that, for every $k\in\N$, 
\[{\dot w}_k=s_k(A+B) \leq M v_k\, ,\] which means $\dot w \preceq v$. 
\end{itemize}
\end{proof}
The proposition above implies that $E_w(H_1, H_2)$ is not a linear space in general. However, it points towards 
a simple sufficient condition guaranteeing linearity.  

\begin{corollary}\label{lin}
If $\dot w\asymp w$, then $E_w(H_1, H_2)$ is a linear space and
${\abs{\cdot}}_w$ is a quasi-norm. 
\end{corollary}
\begin{proof}
By Lemma \ref{al}, we have $\alpha A \in E_w(H_1, H_2)$ for every $\alpha \in \C$ 
and $A \in E_w(H_1, H_2)$.
Moreover, using Proposition \ref{p2} and the assumption $\dot w\asymp w$ we have
\[A + B \in E_{\dot w}(H_1, H_2)=E_w(H_1, H_2)\,.\]
Thus $E_w(H_1, H_2)$ is a linear space. 
It remains to show that ${|\cdot |}_w$ is a quasi-norm. The only non-trivial property is the quasi-triangle inequality, that is, we need to show that there is $M>0$ such that 
\[ {\abs{A+B}}_w \leq M ({\abs{A}}_w+{\abs{B}}_w) \quad (\forall A,B\in E_w(H_1,H_2))\,.\]  
In order to see this note that, since $\dot w\asymp w$ there 
exists $M\geq 1$ such that 
\[
\frac{1}{M} {\abs{A}}_w \leq {\abs{A}}_{\dot w} \leq M{\abs{A}}_w\] 
for every $A \in E_w(H_1, H_2)=E_{\dot w}(H_1, H_2)$. Since $A, B \in E_w(H_1, H_2)$ then,
by Proposition~\ref{p2}, we have  
${\abs{A+B}}_{\dot w} \leq {\abs{A}}_w+{\abs{B}}_w\,$. 
It follows that if $A, B \in E_w(H_1, H_2)$, then $A+B \in E_{\dot w}(H_1, H_2)=E_w(H_1, H_2)$ and 
\[\frac{1}{M}{\abs{A+B}}_w \leq {\abs{A+B}}_{\dot w} \leq {\abs{A}}_w+{\abs{B}}_w\, .\]
 \end{proof}

\begin{proposition} \label{complete}
If $\dot w\asymp w$, then $E_w(H_1, H_2)$ is complete with respect to the quasi-norm ${\abs{\cdot}}_w\,$.
\end{proposition}
\begin{proof}
Let $(A_n)_{n \in \N}$ be a Cauchy sequence in $E_w(H_1, H_2)\,$. First we note that $(A_n)_{n \in \N}$ is a Cauchy sequence in 
$S_{\infty}(H_1, H_2)$ with respect to 
the operator norm $\norm{\cdot}$, since
$\norm{A_n-A_m}\leq {\abs{A_n-A_m}}_w w_1\,$.
As $S_{\infty}(H_1, H_2)$ is complete there is an $A \in S_{\infty}(H_1, H_2)$ 
such that $A_n \rightarrow A$ as $n\rightarrow \infty$ in the operator norm $\norm{\cdot}.$
We need to prove that $A\in E_w(H_1, H_2)$ and ${\abs{A_n-A}}_w\rightarrow 0$ 
as $n\rightarrow \infty\,.$
Fix $\epsilon \geq 0\,.$ Since $(A_n)_{n \in \N}$ is Cauchy in ${\abs{\cdot}}_w\,$, there exists $N_{\epsilon} \in \N$ such that
\[s_k (A_n-A_m) \leq {\abs{A_n-A_m}}_ww_k \leq \epsilon w_k \quad (\forall n, m \geq N_{\epsilon}, 
\forall k \in \N)\, .\]  
Letting $m \rightarrow \infty$ in the above 
we obtain  
\[s_k(A_n-A)\leq \epsilon w_k \quad (\forall n \geq N_{\epsilon}, \forall k \in \N)\,,\]
and so 
\begin{equation} \label{ep}
{\abs{A_n-A}}_w \leq \epsilon \quad (\forall n \geq N_{\epsilon})\,.
\end{equation}
The above implies that ${\abs{A_n-A}}_w\rightarrow 0$ 
as $n\rightarrow \infty$. It remains to show that $A\in E_w(H_1,H_2)$. In order to see this, fix 
$n\geq N_\epsilon$. 
Inequality (\ref{ep}) now implies that 
$A_n-A$ is an element of $E_w(H_1, H_2)$. Since $A_n$ is also an element of $E_w(H_1, H_2)$ and 
$E_w(H_1, H_2)$ is linear by  Corollary \ref{lin}, we then obtain $A\in E_w(H_1, H_2)$. 
\end{proof}
\begin{proposition} \label{quasideal}
If $\dot w\asymp w$, then $E_w$ is a quasi-Banach operator ideal.
\end{proposition}
\begin{proof}
Follows from Proposition~\ref{ideal}, Corollary~\ref{lin} and Proposition~\ref{complete}. 
\end{proof}

We now turn to studying the rate of decay of the eigenvalue sequence of an operator
in a given compactness class. In order to do this we 
require the following notation. 

\begin{definition}
Let $w\in \mathcal W$. Then we define $\bar w$ as the sequence of successive geometric means of $w$, that is, 
\[{\bar w}_k=(w_1\cdots w_k)^{\frac{1}{k}} \quad (\forall k\in \N)\, .\] 
\end{definition}
\begin{proposition} \label{p3}
Let $A\in E_w(H_1, H_2)$. Then 
\[\lambda(A)\in {\mathcal E}_{\bar w}\quad \text{with}\quad {\abs{\lambda(A)}}_{\bar w}\leq {\abs{A}}_w\, .\]
\end{proposition}
\begin{proof}
Let $A\in E_w(H_1, H_2)\,$. 
By the multiplicative Weyl inequality (\ref{twyl1}) we have, for every $k\in\N$, 
  \[{\abs{\lambda_k(A)}}^k\leq\prod_{l=1}^{k}\abs{\lambda_l(A)} \leq \prod_{l=1}^{k}s_l(A)\leq 
  {\abs{A}}_w  w_1\cdots {\abs{A}}_w w_k\leq \abs{A}_w^k w_1\cdots w_k\, .\]
  Thus 
 \[|\lambda_k(A)|\leq {\abs{A}}_w (w_1 \cdots w_k)^{\frac{1}{k}}=|A|_w\bar{w}_k \quad( \forall k\in \N)\,,\]
 and we obtain
 \[\lambda(A)\in {\mathcal E}_{\bar w}\quad \text{and}\quad {\abs{\lambda(A)}}_{\bar w}\leq {\abs{A}}_w\, ,\]
 as desired.
 
\end{proof}

\section{General resolvent bounds}
\label{comsecond}
The first bound for the norm of the resolvent of a linear operator 
on an infinite-dimensional Hilbert space was derived by Carleman (see \cite{Car21}), who 
obtained a bound for Hilbert-Schmidt operators. 
His result was later generalised to Schatten-von
 Neumann operators (see, for example, \cite{DS63, Sim77}). For more information about generalised Carleman
 type estimates see also \cite{DP94, DP96}.
 
In this section we shall derive an upper bound for the norm of the resolvent 
$R(A; z)$ of $A\in E_w(H_1, H_2)$ in terms of the distance of $z$ to the spectrum of $A$ and 
 the $w$-departure from normality of $A$, a number measuring the non-normality of $A$. 
 As already mentioned, we shall generalise the approach of 
 Henrici in \cite{Hen62} outlined in the introduction to the infinite-dimensional setting.
 The basic idea will be to write $A$ as a sum
 of a normal operator $D$ with $\sigma(D)=\sigma(A)$
and a quasi-nilpotent operator $N$, that is, an operator the spectrum of which consists of the point $0$ only, and to consider $A$ as a perturbation of $D$ by $N$. 

We start with a bound for powers of quasi-nilpotent operators, due to 
Dostani\'{c}.  
\begin{theorem}\label{dos}
There is a constant $C\geq\pi/2$ such that for 
any quasi-nilpotent $A\in S_{\infty}(H)$ and
for every $k \in \N$ we have 
\[\norm{A^{2k}}\leq C^{2k}{(s_1(A) \cdots s_k(A))}^2\,.\]

\end{theorem}
 \begin{proof}
 See \cite[Theorem~1]{Dos01}.
 \end{proof}
 
Given $w\in \mathcal W$, we  define a function 
$F_w:\R_0^+\to \R_0^+$  by setting 
\begin{equation}
\label{functionc}
F_w(r)=(1+rw_1)\left(1+\sum_{k=1}^{\infty}(w_1\cdots w_k)^2{(Cr)^{2k}}\right)\,,
\end{equation}
where $C$ is the constant from Theorem~\ref{dos}. 
It is not difficult to see that $F_w$ is well-defined, 
real-analytic and strictly monotonically increasing. 
We are now ready to deduce resolvent bounds for quasi-nilpotent operators.
 \begin{proposition} \label{p5}
 Let $w\in \mathcal{W}$ and let $A\in E_w(H)$ be a quasi-nilpotent operator.  
 Then 
 \[\norm{(I-A)^{-1}}\leq F_w({\abs{A}}_w)\,.\]
 \end{proposition}
 \begin{proof}
 Suppose $A\in E_w(H)$ is a quasi-nilpotent operator. Using a Neumann series and 
 Theorem~\ref{dos}, we have 
 \[\norm{(I-A)^{-1}}\leq \sum_{k=0}^{\infty} \norm{A^{k}}= \sum_{k=0}^{\infty} (\norm{A^{2k}}+\norm{A^{2k+1}})\,,\]
 \[\leq  (1+\norm{A})\left(1+\sum_{k=1}^{\infty}\norm{A^{2k}}\right)\,,\]
 \[\leq  (1+s_1(A))\left(1+\sum_{k=1}^{\infty}C^{2k}(s_1(A) \cdots s_k(A))^2\right)\,.\]
 Therefore we obtain 
 \[\norm{(I-A)^{-1}}\leq (1+{\abs{A}}_ww_1)\left(1+\sum_{k=1}^{\infty}(w_1\cdots w_k)^2(C {\abs{A}}_w)^{2k} \right)\,,\] 
 as required.
 \end{proof}
 An immediate consequence of the previous proposition is the following
 estimate for the growth of the resolvent of a quasi-nilpotent operator $A\in E_w(H)$. 
\begin{corollary}
Let $w\in \mathcal{W}$ and let $A\in E_w(H)$ be quasi-nilpotent. Then for  any $z\not=0$ 
\[\norm{R(A; z)}\leq \abs{z}^{-1}F_w(\abs{z}^{-1}{\abs{A}}_w)\,.\]
\end{corollary}
By means of the following theorem, an upper bound for the norm of the resolvent 
$R(A; z)$ of $A\in E_w(H_1, H_2)$ can be obtained.
\begin{theorem}
\label{schur theorem}
Let $A \in S_{\infty}(H)$. Then $A$ can be written as a sum 
\[A=D+N\,,\] 
such that 
\begin{itemize}
\item[(i)] $D \in S_{\infty}(H), N\in S_{\infty}(H)$; 
\item[(ii)] $D$ is normal and $\lambda(D)=\lambda(A)$;
\item[(iii)] $N$ and $(zI-D)^{-1}N$ are quasi-nilpotent for every $z\in\rho(D)=\rho(A)$.
\end{itemize}
\end{theorem}
\begin{proof}
See \cite[Theorem~3.2]{Ban04}. 
\end{proof}

The theorem above motivates the following definition.  
\begin{definition}
Let $A \in S_{\infty}(H)$. A decomposition
\[A=D+N\] with $D$ and $N$ satisfying the properties (i--iii) of the previous theorem is called a \textit{Schur decomposition of $A$}.
 We call the operators $D$ and $N$ the \textit{normal} and the \textit{quasi-nilpotent part of the Schur decomposition of $A$}, respectively.  
\end{definition}

\begin{remark}
\label{rem:unique}
  The decomposition is not unique, as can be seen from the following example taken from  
  \cite[Remark 3.5 (i)]{Ban04}. Consider
 \begin{align*}
 A:= \left( \begin{array}{ccc}
2 & 2 & 2 \\
0 & 0 & 2 \\
0 & 0 & 0 \end{array} \right) 
& =\underbrace{\left( \begin{array}{ccc}
2 & 0 & 0 \\
0 & 0 & 0 \\
0 & 0 & 0 \end{array} \right)}_{=:D_1} +\underbrace{\left( \begin{array}{ccc}
0 & 2 & 2 \\
0 & 0 & 2 \\
0 & 0 & 0 \end{array} \right)}_{=:N_1} \\
& =\underbrace{\left( \begin{array}{ccc}
1 & 1 & 0 \\
1 & 1 & 0 \\
0 & 0 & 0 \end{array} \right)}_{=:D_2} +\underbrace{\left( \begin{array}{ccc}
1 & 1 & 2 \\
-1 & -1 & 2 \\
0 & 0 & 0 \end{array} \right)}_{=:N_2}\,.
\end{align*}

It is easy to see that $D_1$ and $D_2$ are normal and that $N_1$ and $N_2$ are 
nilpotent. Moreover $\sigma(A)=\sigma(D_1)=\sigma(D_2)=\{2,0\}$. Furthermore, 
both $(zI-D_1)^{-1}N_1$ and  $(zI-D_2)^{-1}N_2$ are nilpotent for any $z\in \rho(A)$. Thus $A$ has two 
different Schur decompositions. 
  
Note that the normal parts are obviously unitarily equivalent. However, 
the nilpotent parts are not. In order to see this observe that  
\[\norm{N_1}_4^4=112\not =80=\norm{N_2}_4^4\,,\] where
$\norm{\cdot}_4$ is the norm of the Schatten class $S_4(\C^3)$. 

\end{remark}
In the following proposition we determine an upper bound for the singular values of the normal part
and the quasi-nilpotent part of a Schur decomposition of an operator in a given compactness class. 
 \begin{proposition}\label{p4}
 Let $A\in E_w(H)$. 
 If $A=D+N$ is a Schur decomposition of $A$ with normal part $D$ and quasi-nilpotent part $N$, then
 \begin{itemize} 
 \item[(i)] $D \in E_{\bar{w}}(H)$ with ${\abs{D}}_{\bar{w}}\leq {\abs{A}}_w\,$, where $\bar{w}_k=(w_1 \cdots w_k)^{\frac{1}{k}}$.
 \item[(ii)] $N \in E_{\dot{\bar{w}}}(H)$ with ${\abs{N}}_{\dot{\bar{w}}}\leq 2{\abs{A}}_w\,$, 
 where $\dot{\bar{w}}=(w_1, w_1, (w_1w_2)^{\frac{1}{2}}, (w_1w_2)^{\frac{1}{2}}, \ldots)\,$.
 \end{itemize}
 \end{proposition}
 \begin{proof}
Let $A\in E_w(H)$. Since $D$ is normal, 
its singular values coincide with the moduli of its eigenvalues, 
which also coincide with the moduli of the eigenvalues of $A$. 
Using Proposition \ref{p3} and the fact that $D$ is normal we obtain 
 \[s_k(D)\leq {\abs{A}}_w \bar {w}_k\,,\]
 so $D\in E_{\bar{w}}(H)$ and 
 ${\abs{D}}_{\bar{w}}\leq {\abs{A}}_w$, as required.
 
 For the second part, observe that since $w \preceq \bar w$, we have ${\abs{A}}_{\bar w}\leq {\abs{A}}_w$ via Proposition \ref{p0}. 
Then we also have $A \in E_{\bar w}(H)$ by Proposition \ref{p1}. 
Thus, using Proposition \ref{p2}, we have  
\[N= A-D \in E_{\dot{\bar{w}}}\, ,\]
  \[{\abs{A-D}}_{\dot{\bar{w}}}\leq{\abs{A}}_{\bar w}+{\abs{D}}_{\bar w}\] 
  and so, using assertion (i), we obtain 
 \[{\abs{N}}_{\dot{\bar w}}\leq {\abs{A}}_w+{\abs{A}}_w= 2{\abs{A}}_w\,,\]
as desired. 
 \end{proof} 
  
 We now define the analogue of Henrici's departure from normality for operators in a given compactness class.
 \begin{definition}
 Let $w\in \mathcal W$ and $A\in E_w(H)$. 
 Then 
 \[\nu_w(A)=\inf \all{\abs{N}_{\dot{\bar w}}}{\text{$N$ is the quasi-nilpotent part of a Schur decomposition of 
 $A$}} \]
is called the \textit{$w$-departure from normality of $A$}. 
 \end{definition}

 \begin{remark} Note that by the previous proposition, the $w$-departure from normality of an operator in $E_w(H)$ is always finite. 
 \end{remark}
 
 The term `departure from normality' is justified in view of the following proposition.
 \begin{proposition}
 Let $A\in E_w(H)$. Then 
 \[A \, \text{is normal} \iff \nu_w(A)=0\,.\]
 \end{proposition}
 \begin{proof}
 The forward implication is trivial. For the backwards implication, let 
 $\nu_w(A)=0$. Then there exists a sequence of Schur decompositions with quasi-nilpotent parts $N_n$
 such that ${\abs{N_n}}_{\dot{\bar w}}\rightarrow 0$ as $n\rightarrow \infty$. But
 \[\norm{A-D_n}=\norm{N_n}=s_1(N_n)\leq {\dot{\bar w}}_1 {\abs{N_n}}_{\dot{\bar w}}\,,\]
 where $D_n$ are the corresponding normal parts, so $\lim_{n\to\infty}\norm{A-D_n}=0$.
  Hence $A$ is a limit of normal 
 operators which converge in operator norm. Thus $A$ is normal. 
 \end{proof}
 Since the departure from normality is difficult to calculate for a given $A\in E_w(H)$,
 we now give a simple upper bound. 
 \begin{proposition}
 \label{lw}
 Let $A\in E_w(H)$. Then 
 \[\nu_w(A)\leq 2\abs{A}_w\,.\]
 \end{proposition}
 \begin{proof}
Follows from Proposition~\ref{p4} (ii). 
 \end{proof}
 We are now able to obtain growth estimates for the resolvents of operators in a given compactness class. 
 Before doing so we recall the bound for the resolvent of a normal operator.
 If $D$ is a normal operator on a separable Hilbert space, then 
 \begin{equation}\label{equa1}
 \norm{R(D; z)}=\frac{1}{d(z, \sigma(D))}\quad (\forall z\in\rho(D))\,.
 \end{equation}
 The following is the main result of this section. 
 \begin{theorem}\label{t1}
Let $A\in E_w(H)$. Then 
\begin{equation}
\label{et1}
\norm{R(A; z)}\leq \frac{1}{d(z, \sigma(A))}F_{\dot{\bar w}}\left(\frac{\nu_w(A)}{d(z, \sigma(A))}\right)
\quad (\forall z\in \rho(A))\,.
\end{equation}
\end{theorem}
 \begin{proof}
 Fix $z\in \rho(A)$. By Proposition \ref{p4}, the operator $A$ has a Schur decomposition 
 with normal part $D$ and quasi-nilpotent part $N$. 
 Thus, we know that $\sigma(A)=\sigma(D)$, that $(zI-D)^{-1}$ exists 
 and that $(zI-D)^{-1}N$ is quasi-nilpotent. Furthermore
\begin{multline*}
 s_k((zI-D)^{-1}N)
 \leq \norm{(zI-D)^{-1}}s_k(N)\\
=\frac{s_k(N)}{d(z, \sigma(D))}
 \leq \frac{{\abs{N}}_{\dot{\bar w}} {\dot{\bar w}}_k}{d(z, \sigma(D))}
 =\frac{{\abs{N}}_{\dot{\bar w}} {\dot{\bar 
 w}}_k}{d(z, \sigma(A))}\,,
 \end{multline*} 
 using (\ref{equa1}) as well as (\ref{peqo}) and Proposition \ref{p4}. 
 Now $(I-(zI-D)^{-1}N)$ is invertible in $L(H)$ and, using Proposition \ref{p5}, it follows that 
 \[
 \norm{(I-(zI-D)^{-1}N)^{-1}}\leq F_{\dot{\bar w}}\left(\frac{{\abs{N}}_{\dot{\bar w}}}{d(z, \sigma(A))}\right)\,. 
 \]
 Since 
 \[(zI-A)=(zI-D)(I-(zI-D)^{-1}N)\,,\] we can conclude that $(zI-A)$ is invertible in $L(H)$ and 
  \begin{align*}
\norm{R(A; z)}
 &\leq \norm{R(D; z)}\norm{(I-(zI-D)^{-1}N)^{-1}}\\
 &\leq \frac{1}{d(z, \sigma(A))}F_{\dot{\bar w}}\left(\frac{{\abs{N}}_{\dot{\bar w}}}{d(z, \sigma(A))}\right)\,.
 \end{align*}
 Taking the infimum over all Schur decompositions the theorem follows. 
 \end{proof}
 \begin{remark}
 \label{rem:upper}\
 \begin{itemize} 
  \item[(i)] Another look at the above proof shows that the bound (\ref{et1}) also holds if we replace 
  $\nu_w(A)$ by a larger quantity, say by the upper bound given in Proposition~\ref{lw}. 
 \item[(ii)] The bound (\ref{et1}) is optimal for normal $A$, as it reduces to the
 sharp bound (\ref{equa1}). 
 \end{itemize}
 \end{remark}

\section{Resolvent bounds for particular classes}
\label{sec:particular}
As we saw in the last section, the growth of the resolvent of an operator
belonging to a given compactness class $E_w$ in the vicinity
of a spectral point is, by Theorem~\ref{t1}, 
controlled by the behaviour of the function 
$F_{\dot{\bar{w}}}$ at infinity. In this section we shall 
study the asymptotics of this function for particular compactness
classes, namely the Schatten-Lorentz ideals, given by $w_k=k^{-1/p}$
with $p\in(0,\infty)$ and the exponential classes, given by 
$w_k=\exp(-ak^\alpha)$ with $a\in (0,\infty)$ and $\alpha\in
(0,\infty)$. 

Before starting with the Schatten-Lorentz ideals we briefly recall
Stirling's approximation for the factorial in the form
\[ \sqrt{2\pi k} \left ( \frac{k}{\e} \right )^k \leq k! \leq
\sqrt{\e^2k}  \left ( \frac{k}{\e} \right )^k \quad (\forall k\in\N)\,.\]

\begin{lemma}
\label{lem:SLwdb}
Let $p\in(0,\infty)$, and let $w_k=k^{-1/p}$ for $k\in \N$. Then the
following inequalities hold: 
\begin{equation}
\label{lem:SLwdbe1}
\exp\left (-\frac{1}{p\sqrt{k}} \right ) \frac{\e^{1/p}}{k^{1/p}} 
\leq \bar{w}_k \leq  \frac{\e^{1/p}}{k^{1/p}}\quad (\forall k\in \N)\,,
\end{equation} 
\begin{equation}
\label{lem:SLwdbe2}
\exp\left (-\frac{3}{p\sqrt{k}} \right ) \frac{(2\e)^{1/p}}{k^{1/p}} 
\leq \dot{\bar{w}}_k \leq  
\frac{(2\e)^{1/p}}{k^{1/p}}\quad (\forall k\in \N)\,,
\end{equation} 
\begin{equation}
\label{lem:SLwdbe3}
\exp\left (-\frac{6}{p}\sqrt{k} \right )
\frac{(2\e)^{1/p}}{(k!)^{1/p}}
\leq  
\prod_{n=1}^k \dot{\bar{w}}_n \leq \frac{(2\e)^{1/p}}{(k!)^{1/p}}
\quad (\forall k\in \N)\,.
\end{equation}
\end{lemma}
\begin{proof}
We start with the case $p=1$, that is, we set $w_k=k^{-1}$ and show that 
\begin{equation}
\label{nop1}
\exp\left (-\frac{1}{\sqrt{k}} \right ) \frac{\e}{k} 
\leq \bar{w}_k \leq  \frac{\e}{k}\quad (\forall k\in \N)\,,
\end{equation} 
\begin{equation}
\label{nop2}
\exp\left (-\frac{3}{\sqrt{k}} \right ) \frac{2\e}{k} 
\leq \dot{\bar{w}}_k \leq  \frac{2\e}{k}\quad (\forall k\in \N)\,,
\end{equation} 
\begin{equation}
\label{nop3}
\exp\left (-6\sqrt{k} \right ) \frac{2\e}{k!} 
\leq \prod_{n=1}^k\dot{\bar{w}}_n \leq  \frac{2\e}{k!}\quad (\forall k\in \N)\,.
\end{equation} 

Now, the upper bound in (\ref{nop1}) 
follows from Stirling's approximation by observing
that for all $k\in \N$ we have 
\[ \bar{w}_k^k=\frac{1}{k!}\leq \left ( \frac{\e}{k}\right )^k \,.\]
For the lower bound in (\ref{nop1}) we again use Stirling's
approximation to obtain 
\[ \bar{w}_k^k=\frac{1}{k!}\geq \frac{1}{\sqrt{\e^2k}}  
\left ( \frac{\e}{k}\right )^k\,,\]
and we see that we are done if we can show that 
\begin{equation}
\label{nopaux}
 \frac{1}{\sqrt{\e^2k}} \geq \exp(-\sqrt{k}) \quad (\forall k\in
 \N)\,. 
\end{equation}
The above, however, is true since, using the inequality
$1+x\leq \exp(x)$ which holds for all real x, we see that for all
$k\in \N$ we have 
\[ \sqrt{k} \leq \exp( \sqrt{k}-1 ) \]
from which 
\[ \sqrt{\e^2 k} \leq \exp( \sqrt{k} )\,, \]
which implies (\ref{nopaux}).  

We now turn to (\ref{nop2}). 
For the upper bound we note that, for 
$k\in \N$ even, (\ref{nop1}) implies 
\[ \dot{\bar{w}}_k=\bar{w}_{\frac{k}{2}} \leq \frac{2\e}{k}\,, \]
while for $k\in \N$ odd, (\ref{nop1}) implies 
\[ \dot{\bar{w}}_k=\bar{w}_{\frac{k+1}{2}} \leq \frac{2\e}{k+1}
\leq \frac{2\e}{k}\,. \]
For the lower bound
we note that, for 
$k\in \N$ even, (\ref{nop1}) implies 
\[ \dot{\bar{w}}_k=\bar{w}_{\frac{k}{2}} 
\geq \exp \left (- \frac{\sqrt{2}}{\sqrt{k}}\right ) \frac{2\e}{k}
\geq \exp \left (- \frac{3}{\sqrt{k}}\right ) \frac{2\e}{k}\,, \]
while for $k\in \N$ odd, (\ref{nop1}) implies
\[ \dot{\bar{w}}_k=\bar{w}_{\frac{k+1}{2}} 
\geq \exp \left ( -\frac{\sqrt{2}}{\sqrt{k+1}}\right )
\frac{2\e}{k+1}\,, \]
and we are done if we can show that for all $k\in \N$ we have 
\[ \exp \left (-\frac{\sqrt{2}}{\sqrt{k+1}} \right) \frac{1}{k+1} \geq 
\exp \left (-\frac{3}{\sqrt{k}} \right) \frac{1}{k}\,,\]
which, in turn, is equivalent to 
\begin{equation}
\label{nopaux2}
\left (1+ \frac{1}{k} \right ) 
\exp \left (-\frac{3}{\sqrt{k}} +\frac{\sqrt{2}}{\sqrt{k+1}} 
\right ) \leq 1 \quad (\forall
k\in \N)\,.
\end{equation}
The above, however, follows by observing that we have for all $k\in
\N$ 
\begin{multline*} 
\left (1+ \frac{1}{k} \right ) 
\exp \left (-\frac{3}{\sqrt{k}} +\frac{\sqrt{2}}{\sqrt{k+1}} \right ) 
\leq 
\left (1+ \frac{1}{k} \right ) 
\exp \left (-\frac{1}{\sqrt{k}} \right ) 
\\
\leq \exp \left (\frac{1}{k}-\frac{1}{\sqrt{k}} \right ) 
\leq \exp \left (-\frac{\sqrt{k}-1}{k} \right ) 
\leq 1\,.
\end{multline*}
This finishes the proof of (\ref{nop2}). 

Finally, the upper bound in (\ref{nop3}) is obvious, while the lower 
one follows
from 
\[  \prod_{n=1}^k\dot{\bar{w}}_n\geq 
\exp \left ( -3\sum_{n=1}^k\frac{1}{\sqrt{n}}\right ) 
\frac{(2\e)^{k}}{k!}
\geq 
\exp \left ( -6 \sqrt{k} \right ) 
\frac{(2\e)^{k}}{k!}\,,
\]
where we have used that $\sum_{n=1}^kn^{-1/2} \leq \int_0^k t^{-1/2}
=2k^{1/2}$ for every $k\in \N$. 

This finishes the proof of the lemma for $p=1$. 
The general case follows by taking $p$-th roots in
(\ref{nop1}), (\ref{nop2}) and (\ref{nop3}). 
\end{proof}

In order to be able to study the behaviour of $F_{\dot{\bar{w}}}$ we
require another auxiliary result. Before stating it we introduce some
more notation. If 
$f$ and $g$ are two real-valued functions defined on a neighbourhood
of $\infty$, we write 
\[ f(r) \sim g(r) \text{ as $r\to \infty$} \]
if 
\[ \lim_{r\to \infty}\frac{f(r)}{g(r)}=1\,.\]
For later use, we note the following relation between the asymptotics
of a function and that of its inverse. 
\begin{lemma}
\label{lem:genasym}
Let $a,b\in (0,\infty)$ and let $I$ and $J$ be neighbourhoods of
$\infty$. Suppose that $f:I\to J$ is a bijection with inverse
$f^{-1}:J\to I$. Then the following assertions hold. 
\begin{itemize}
\item[(i)]
If 
\[ f(r)\sim ar^b \text{ as $r\to \infty$} \]
then 
\[ f^{-1}(r) \sim \left ( \frac{r}{a} \right )^{1/b} \text{ as $r\to
  \infty$}\,. \]
\item[(ii)] If 
\[ \log f(r)\sim ar^b \text{ as $r\to \infty$} \]
then 
\[ f^{-1}(r) \sim \left ( \frac{\log r}{a} \right )^{1/b} \text{ as
  $r\to \infty$}\,. \]
\item[(iii)] If 
\[ \log f(r)\sim a(\log r)^b \text{ as $r\to \infty$} \]
then 
\[ \log f^{-1}(r) \sim \left ( \frac{\log r}{a} \right )^{1/b} \text{ as $r\to \infty$}\,. \]
\end{itemize}
\end{lemma}
\begin{proof} \ 
\begin{itemize}
\item[(i)]
This follows from 
\[ \lim_{r\to \infty} \frac{(r/a)^{1/b}}{f^{-1}(r)}=
\lim_{r\to \infty}\frac{(f(r)/a)^{1/b}}{f^{-1}(f(r))}=
\left ( \lim_{r\to \infty} \frac{f(r)}{ar^b}\right )^{1/b}=1\,.\]
\item[(ii)]
If 
\[ (\log \circ f)(r) \sim a r^b  
\text{ as $r\to \infty$}\,,\]
then by (i) we have 
\[ (\log \circ f)^{-1}(r) \sim \left ( \frac{r}{a} \right )^{1/b}
\text{ as $r\to \infty$}\,,\]
so 
\[ (f^{-1}\circ \exp)(r) \sim \left ( \frac{r}{a} \right )^{1/b}
\text{ as $r\to \infty$}\,,\]
hence 
\[ f^{-1}(r) \sim \left ( \frac{\log r}{a} \right )^{1/b}
\text{ as $r\to \infty$}\,.\]
\item[(iii)]
If 
\[ (\log \circ f)(r) \sim a (\log r)^b  
\text{ as $r\to \infty$}\,,\]
then 
\[ (\log \circ f\circ \exp )(r) \sim a r^b  
\text{ as $r\to \infty$}\,,\]
so by (i) we have 
\[ (\log \circ f \circ \exp )^{-1}(r) \sim \left ( \frac{r}{a} \right )^{1/b}
\text{ as $r\to \infty$}\,,\]
hence 
\[ (\log \circ f^{-1}\circ \exp)(r) \sim \left ( \frac{r}{a} \right )^{1/b}
\text{ as $r\to \infty$}\,,\]
whence 
\[ (\log \circ f^{-1})(r) \sim \left ( \frac{\log r}{a} \right )^{1/b}
\text{ as $r\to \infty$}\,.\]
\end{itemize}
\end{proof}

We are now able to state the following result. 
\begin{lemma}
\label{lem:SLgrowth}
Suppose that $p,b\in (0,\infty)$. 
Let $\Phi_p^{L,u}$ and  $\Phi_{p,b}^{L,l}$ be two 
functions
given by the power series 
\[ \Phi_p^{L,u}(r)=\sum_{k=0}^\infty \frac{1}{(k!)^{1/p}}r^k \]
\[ \Phi_{p,b}^{L,l}(r)=\sum_{k=0}^\infty
\frac{\exp(-b\sqrt{k})}{(k!)^{1/p}}r^k\,. 
\] 
Then
$\Phi_p^{L,u}$ and $\Phi_{p,b}^{L,l}$ extend to entire functions with 
the following asymptotics
\[ \log \Phi_p^{L,u}(r)\sim \log \Phi_{p,b}^{L,l}(r) \sim \frac{1}{p}r^p
\text{ as $r\to\infty$}\,.\]
\end{lemma}
\begin{proof}
Using Stirling's approximation we see that both $\Phi_p^{L,u}$ and
$\Phi_{p,b}^{L,l}$ extend to entire functions. Since 
$\Phi_{p,b}^{L,l}(r) \leq \Phi_p^{L,u}(r)$ for all $r\in (0,\infty)$
the remaining assertions will hold if we can show that 
\begin{equation}
\label{SLasympsup}
\limsup_{r\to \infty} pr^{-p}\log \Phi_p^{L,u}(r) \leq 1 
\end{equation} 
and 
\begin{equation}
\label{SLasympinf}
\liminf_{r\to \infty} pr^{-p}\log \Phi_{p,b}^{L,l}(r) \geq 1\,. 
\end{equation} 
We start with (\ref{SLasympsup}). For $p\leq 1$ we have, using the
$\ell_p$-$\ell_1$ inequality 
\[ \sum_{k=0}^\infty x_k \leq \left ( \sum_{k=0}^\infty x_k^p \right
)^{1/p}\,,\]
which holds for all positive sequences $(x_k)_{k=0}^\infty$, the bound 
\[ \Phi_p^{L,u}(r)\leq \left ( \sum_{k=0}^\infty 
   \frac{r^{pk}}{k!}\right ) ^{1/p}=
\exp\left ( \frac{1}{p}r^p \right )\,,\]
and (\ref{SLasympsup}) holds in this case. For $p>1$ we split the sum
as follows 
\[  \Phi_p^{L,u}(r) = \sum_{k<2\e r^p} \frac{r^k}{(k!)^{1/p}} + 
\sum_{k\geq 2\e r^p} \frac{r^k}{(k!)^{1/p}}\,.\]
In order to bound the first term we use H\"older's inequality to
obtain 
\begin{align*}
\sum_{k<2\e r^p} \frac{r^k}{(k!)^{1/p}} 
& \leq \left ( \sum_{k<2\e r^p} \frac{r^{pk}}{k!} \right )^{1/p} 
\left ( \sum_{k<2\e r^p} 1 \right )^{(p-1)/p} \\
& \leq (1+2\e r^p)^{(p-1)/p}\exp \left (\frac{1}{p} r^p\right )\,.
\end{align*} 
For the second term, we use Stirling's approximation and obtain 
\[ \sum_{k\geq 2\e r^p} \frac{r^k}{(k!)^{1/p}}
\leq \sum_{k\geq 2\e r^p} \left ( \frac{\e r^p}{k} \right )^{k/p}
\leq  \sum_{k\geq 2\e r^p} 2^{-k/p} \leq 2\cdot 2^{-(2\e r^p)/p}\,.\]
Combining these two estimates, the bound (\ref{SLasympsup}) follows for
$p> 1$ as well. 

We now turn to the proof of (\ref{SLasympinf}). For a given $r\geq 1$
choose $k\in \N$ such that 
\[ r^p-1<k\leq r^p\,.\]
Since all terms in the sum defining $\Phi_{p,b}^{L,l}$ are positive it 
follows that 
\[ \Phi_{p,b}^{L,l}(r) \geq \frac{\exp(-b\sqrt{k})}{(k!)^{1/p}}r^k\,.\]
Now 
\[ r^k\geq r^{r^p-1} \]
and, using Stirling's approximation, 
\[ (k!)^{1/p}\leq (\e^2 k)^{1/(2p)} 
\left ( \frac{k}{\e} \right )^{k/p} 
\leq 
( \e^2 r^p )^{1/(2p)} \frac{r^{r^p}}{\e^{\frac{1}{p}r^p}}\,.
\]
Furthermore, we have 
\[ \exp(-b\sqrt{k}) \geq \exp( -b r^{p/2})\,.\]
Thus, combining all previous estimates and simplifying we have 
\[  \Phi_p^{L,u}(r) \geq \frac{\exp(-br^{p/2})}{\e^{1/p}r^{3/2}} \left (
\frac{1}{p}r^p\right )\,,
\] 
and the bound (\ref{SLasympinf}) follows.  
\end{proof}

We are now ready to give upper and lower bounds for 
$F_{\dot{\bar{w}}}$ as well as its asymptotics 
for $w$ generating the Schatten-Lorentz ideal. 

\begin{proposition}
\label{prop:SLgrowth}
Let $p\in (0,\infty)$ and let $w_k=k^{-1/p}$ for $k\in \N$. Then for
all $r>0$ we have 
\begin{equation}
\label{propPhieq}
(1+r)\Phi_{p/2,12/p}^{L,l}\left ( (2\e)^{2/p}
  (Cr)^2 \right ) 
\leq 
F_{\dot{\bar{w}}}(r) \leq (1+r)
\Phi_{p/2}^{L,u}\left ( (2\e)^{2/p}
  (Cr)^2 \right ) 
\end{equation}
Moreover 
\[ \log F_{\dot{\bar{w}}}(r) \sim \frac{4\e C^p}{p}r^p \text{ as $r\to
\infty$}.  
\]
\end{proposition}
\begin{proof}
By Lemma~\ref{lem:SLwdb} we have for all $k\in \N$
\[ 
\exp \left ( -\frac{12}{p} \sqrt{k} \right ) 
\frac{(2\e)^{2k/p}}{(k!)^{2/p}}
\leq 
\prod_{n=1}^k\dot{\bar{w}}_n^2\leq \frac{(2\e)^{2k/p}}{(k!)^{2/p}}\,.
  \]
Using the definition of $F_{\dot{\bar{w}}}$ in (\ref{functionc})
the inequalities in 
(\ref{propPhieq}) follow, which, using Lemma~\ref{lem:SLgrowth}, imply 
the remaining assertion.  
\end{proof}

We now turn our attention to the exponential cases, which are
compactness classes $E_w$ with weights of the form
$w_k=\exp(-ak^\alpha)$ with $a,\alpha\in (0,\infty)$. We start with
two technical lemmas. 

\begin{lemma}
\label{lem:Ewdb}
Let $a,\alpha \in (0,\infty)$, and let $w_k=\exp(-ak^\alpha)$ for $k\in
\N$. Then there are strictly positive real constants
$\bar{c}_{a,\alpha}$, 
$\dot{\bar{c}}_{a,\alpha}$ and  
$c_{a,\alpha}$ such that the following 
inequalities hold for every $k\in \N$
\begin{equation}
\label{lem:Ewdbe1}
\exp\left (-\frac{a}{\alpha +1}k^\alpha - 
\bar{c}_{a,\alpha} k^{\alpha-1/2} \right )
\leq \bar{w}_k \leq  
\exp\left (-\frac{a}{\alpha +1}k^\alpha \right )\,,
\end{equation}
 \begin{equation}
\label{lem:Ewdbe2}
\exp\left (-\frac{2^{-\alpha}a}{\alpha +1}k^\alpha - 
\dot{\bar{c}}_{a,\alpha} k^{\alpha-1/2} \right )
\leq \dot{\bar{w}}_k \leq  
\exp\left (-\frac{2^{-\alpha}a}{\alpha +1}k^\alpha \right )\,,
\end{equation} 
\begin{equation}
\label{lem:Ewdbe3}
\exp\left (-\frac{2^{-\alpha}a}{(\alpha +1)^2}k^{\alpha+1}- 
c_{a,\alpha} k^{\alpha+1/2} \right )
\leq \prod_{n=1}^k\dot{\bar{w}}_n \leq  
\exp\left (-\frac{2^{-\alpha}a}{(\alpha +1)^2}k^{\alpha+1} \right
)\,. 
\end{equation} 
\end{lemma}
\begin{proof}
We start with (\ref{lem:Ewdbe1}). First we note that 
\[ \bar{w}_k^k =\exp \left ( -a \sum_{n=1}^kn^\alpha \right )
\quad (\forall k\in \N)\,. \]
Since 
$\int_0^kt^\alpha\,dt \leq \sum_{n=1}^kn^\alpha \leq 
\int_0^{k+1} t^\alpha\,dt$,  
we have 
\begin{equation}
\label{eq:alphasum}
\frac{1}{\alpha +1}k^{\alpha+1}\leq    \sum_{n=1}^kn^\alpha 
\leq 
\frac{1}{\alpha +1}(k+1)^{\alpha+1} \quad
(\forall k\in \N)\,,
\end{equation}
from which the upper bound of (\ref{lem:Ewdbe1}) readily follows,
while the lower bound can be obtained by observing that there is a constant
$K_1>0$ such that 
\[ \frac{(k+1)^{\alpha+1}}{k}\leq k^\alpha +K_1 k^{\alpha -1/2}
\quad (\forall k \in \N)\,.\]

For the next pair of inequalities (\ref{lem:Ewdbe2}) we note that, 
using (\ref{lem:Ewdbe1}), we have  
for $k\in \N$ even 
\[ \dot{\bar{w}}_k= \bar{w}_{\frac{k}{2}} \leq 
  \exp\left (-\frac{2^{-\alpha}a}{\alpha +1}k^\alpha \right )\,, \]
while for $k\in \N$ odd 
\[ \dot{\bar{w}}_k= \bar{w}_{\frac{k+1}{2}} \leq 
  \exp\left (-\frac{2^{-\alpha}a}{\alpha +1}(k+1)^\alpha \right ) 
\leq   
\exp\left (-\frac{2^{-\alpha}a}{\alpha +1}k^\alpha \right )\,,
\]
and the upper bound follows. For the lower bound we note that by 
(\ref{lem:Ewdbe1}), we have  
for $k\in \N$ even 
\begin{equation*}
 \dot{\bar{w}}_k = \bar{w}_{\frac{k}{2}}
  \geq  
  \exp\left (-\frac{2^{-\alpha}a}{\alpha +1}k^\alpha 
   -2^{-\alpha+1/2}\bar{c}_{a,\alpha} k^{\alpha -1/2}
\right )\,, 
\end{equation*}
while for $k\in \N$ odd we have 
\begin{equation*}
 \dot{\bar{w}}_k = \bar{w}_{\frac{k+1}{2}} 
  \geq  
  \exp\left (-\frac{2^{-\alpha}a}{\alpha +1}(k+1)^\alpha 
   -2^{-\alpha+1/2}\bar{c}_{a,\alpha} (k+1)^{\alpha -1/2}
\right )\,, 
\end{equation*}
from which the lower bound follows for all $k\in \N$ by observing that
for any $\beta > 0$ and any $K_2>0$ there is a constant $K_3>0$ such that 
\begin{equation}
\label{Kaux}
(k+1)^{\beta} + K_2(k+1)^{\beta -1/2}
\leq k^\beta +K_3 k^{\beta -1/2}
\quad (\forall k \in \N)\,.
\end{equation}
Finally, using (\ref{lem:Ewdbe2}) and (\ref{eq:alphasum}), 
the upper bound in (\ref{lem:Ewdbe3}) follows, since we have for all $k\in \N$ 
\[ \prod_{n=1}^k\dot{\bar{w}}_n \leq 
\exp \left (  - \frac{2^{-\alpha}a}{\alpha +1} \sum_{n=1}^k n^\alpha
\right ) 
\leq 
\exp \left (  - \frac{2^{-\alpha}a}{(\alpha +1)^2} k^{\alpha+1} 
\right )\,. 
\]
The lower bound in turn follows from 
\begin{align*}
\prod_{n=1}^k\dot{\bar{w}}_n & \geq 
\exp \left (  - \frac{2^{-\alpha}a}{\alpha +1} \sum_{n=1}^k n^\alpha
- \dot{\bar{c}}_{a,\alpha}\sum_{n=1}^k n^{\alpha-1/2}
\right ) \\
& \geq 
\exp \left (  - \frac{2^{-\alpha}a}{(\alpha +1)^2} (k+1)^{\alpha+1}
- \frac{2\dot{\bar{c}}_{a,\alpha}}{2\alpha +1} (k+1)^{\alpha+1/2}
\right )
\end{align*}
and (\ref{Kaux}). 
\end{proof}

\begin{lemma}
\label{lem:Egrowth}
Suppose that $a,\alpha,b\in (0,\infty)$. 
Let $\Phi_{a,\alpha}^{E,u}$ and  $\Phi_{a,\alpha,b}^{E,l}$ be two 
functions
given by the power series 
\begin{equation}
\label{PhiEudef}
\Phi_{a,\alpha}^{E,u}(r)=\sum_{k=0}^\infty \exp( -ak^{\alpha+1})r^k\,,
\end{equation}
\begin{equation}
\label{PhiEldef}
\Phi_{a,\alpha,b}^{E,l}(r)=\sum_{k=0}^\infty
\exp(-ak^{\alpha+1} -bk^{\alpha+1/2}) r^k\,. 
\end{equation} 
Then
$\Phi_{a,\alpha}^{E,u}$ and $\Phi_{a,\alpha,b}^{E,l}$ 
extend to entire functions with 
the following asymptotics
\[ \log \Phi_{a,\alpha}^{E,u}(r)\sim 
\log \Phi_{a,\alpha,b}^{E,l}(r) \sim 
a^{-1/\alpha} \frac{\alpha}{(\alpha+1)^{1+1/\alpha}}
\left ( \log r \right )^{1+1/\alpha}
\text{ as $r\to\infty$}\,.\]
\end{lemma}
\begin{proof}
It is not difficult to see that both $\Phi_{a,\alpha}^{E,u}$ and
$\Phi_{a,\alpha,b}^{E,l}$ extend to entire functions. As in the proof
of the analogous result for the Schatten-Lorentz ideal, we note that since 
$\Phi_{a,\alpha,b}^{E,l}(r) \leq \Phi_{a,\alpha}^{E,u}(r)$ for all 
$r\in (0,\infty)$, 
the remaining assertions will hold if we can show that 
\begin{equation}
\label{Easympsup}
\limsup_{r\to \infty}
a^{1/\alpha} \frac{(\alpha+1)^{1+1/\alpha}}{\alpha}
\left ( \log r \right )^{-1-1/\alpha} 
\log \Phi_{a,\alpha}^{E,u}(r) \leq 1 
\end{equation} 
and 
\begin{equation}
\label{Easympinf}
\liminf_{r\to \infty} 
a^{1/\alpha} \frac{(\alpha+1)^{1+1/\alpha}}{\alpha}
\left ( \log r \right )^{-1-1/\alpha} 
\log \Phi_{a,\alpha,b}^{E,l}(r) \geq 1\,. 
\end{equation} 
We start with (\ref{Easympsup}). Fix $r\geq 1$. Let $\mu(r)$ denote
the maximal term of the series (\ref{PhiEudef}), that is, 
\[ \mu(r)=\max_{k\in \N} \left \{\exp(-ak^{\alpha+1})r^k\right \}\,,\]
and note that 
\begin{equation}
\label{expmax}
\log \mu(r)\leq a^{-1/\alpha} \frac{\alpha}{(\alpha+1)^{1+1/\alpha}}
\left ( \log r \right )^{1+1/\alpha}\,,
\end{equation}
which follows from a short calculation. 
Next, let 
\[ k(r)= \left ( \frac{\log(2r)}{a} \right )^{1/\alpha}\,,\]
and observe that 
\[ \exp(-ak^{\alpha+1})\leq (2r)^{-k} \quad (\forall k\geq k(r))\,. \]
Thus, for every $r\geq 1$ we have 
\begin{align*}
\Phi_{a,\alpha}^{E,u}(r)&=
 \sum_{k<k(r)} \exp(-ak^{\alpha+1})r^k+
 \sum_{k\geq k(r)} \exp(-ak^{\alpha+1})r^k\\
 & \leq \sum_{k<k(r)} \mu(r) 
   + \sum_{k\geq k(r)} \frac{1}{2^k} \\
 & \leq (k(r)+1)\mu(r) + \frac{1}{2^{k(r)}}\,,
\end{align*}
from which (\ref{Easympsup}) follows. 

We now turn to the proof of (\ref{Easympinf}). For a given $r\geq 1$
choose $k\in \N$ such that 
\[ k\leq \left ( \frac{\log r}{a(\alpha+1)} \right )^{1/\alpha}< k+1\,.\]
Since all terms in the sum defining $\Phi_{a,\alpha,b}^{E,l}$ are
positive we have 
\begin{align*}
\Phi_{a,\alpha,b}^{E,l}(r) 
   &\geq \exp(-ak^{\alpha+1} -bk^{\alpha+1/2})r^k \\
& \geq \frac{1}{r}\exp\left (-b \left (\frac{\log r}{a(\alpha+1)} \right )^{1+1/(2\alpha)} \right )
\exp\left (\frac{\alpha \left ( \log r \right
   )^{1+1/\alpha} }{a^{1/\alpha}(\alpha+1)^{1+1/\alpha}}
 \right)\,, 
\end{align*}
from which the bound (\ref{Easympinf}) follows.  
\end{proof}

We are now able to give upper and lower bounds as well as
the precise asymptotics of $F_{\dot{\bar{w}}}$ for weights generating
exponential classes.

\begin{proposition}
\label{prop:Egrowth}
Let $a,\alpha \in (0,\infty)$ and let $w_k=\exp(-ak^\alpha)$ 
for $k\in \N$. Then for
all $r\geq 1$ we have 
\begin{equation}
\label{propEPhieq}
(1+r)\Phi_{a',\alpha,2c_{a,\alpha}}^{E,l}\left( (Cr)^2 \right ) 
\leq 
F_{\dot{\bar{w}}}(r) \leq (1+r)
\Phi_{a',\alpha}^{E,u}\left ( (Cr)^2 \right )\,, 
\end{equation}
where $c_{a,\alpha}$ is the constant occurring in Lemma~\ref{lem:Ewdb} and 
\[ a'=\frac{2^{1-\alpha}a}{(\alpha+1)^2}\,.\]
Moreover 
\[ \log F_{\dot{\bar{w}}}(r) \sim 
4 \left ( \frac{\alpha+1}{a}\right )^{1/\alpha}\frac{\alpha}{\alpha+1} 
\left ( \log r \right )^{1+1/\alpha}  
\text{ as $r\to
\infty$}.  
\]
\end{proposition}
\begin{proof}
The inequalities in (\ref{propEPhieq}) follow from 
Lemma~\ref{lem:Ewdb} and the definition of 
$F_{\dot{\bar{w}}}$ in (\ref{functionc}). 
The remaining assertion follows from (\ref{propEPhieq}) and 
Lemma~\ref{lem:Egrowth}. 
\end{proof}
\section{Bounds for the spectral distance}
\label{comthird}
The resolvent bounds deduced
in Section~\ref{comsecond}
together with the Bauer-Fike argument which will be stated below allow us to 
derive the main result of this article: 
upper bounds for the spectral distance of two operators 
belonging to $E_w(H)$ expressible in terms of the distance of the two 
operators in operator norm and their 
$w$-departures from normality. 

The formulation below is based on \cite[Proposition~4.1]{Ban08}. 

\begin{theorem}
\label{otbf}
Let $A\in S_{\infty}(H_1, H_2)$. Suppose that there is a strictly 
monotonically increasing
surjective 
function $g:[0,\infty)\rightarrow [0,\infty)$ and a positive constant $K$ such that
\[
\norm{(zI-A)^{-1}}\leq \frac{1}{K}g\left(\frac{K}{d(z, \sigma(A))}\right)\quad (\forall z \not\in \sigma(A))\,.
\]
Then, for any $B\in L(H_1, H_2)$, we have 
\[\hat d(\sigma(B), \sigma(A))\leq Kh\left ( \frac{\norm{A-B}}{K} \right )\,.\]
 Here, the function $h:[0,\infty)\rightarrow [0,\infty)$ is given by 
 \[h(r)=(\tilde{g}(r^{-1}))^{-1}\,,\]
 where
 $\tilde{g}:[0,\infty)\rightarrow [0,\infty)$ is the inverse of the function $g$.
\end{theorem}
\begin{proof}
Assume $B-A\not =0$, since otherwise there is nothing to prove. 
We start by establishing the following statement:
\begin{equation}
\label{eq:bf60}
\text{if $z\in\sigma(B)$, but $z\not\in\sigma(A)$, then
$\norm{{B-A}}^{-1}\leq \norm{(zI-A)^{-1}}$\,.}
\end{equation}
This is done  
by contradiction. 
Let $z\in\sigma(B)$ and $z\not\in\sigma(A)$. Assume to the contrary that 
 \[\norm{(zI-A)^{-1}}\norm{B-A} < 1\,.\]
Then $(I-{(zI-A)}^{-1}(B-A))$ is invertible. It follows that \[(zI-B)=(zI-A)(I-(zI-A)^{-1}(B-A))\] is invertible. 
Therefore  $z\not\in\sigma(B)$ which contradicts $z\in \sigma(B)$. Hence statement (\ref{eq:bf60}) holds.

In order to prove the theorem it suffices to show that 
if $z\in \sigma(B)$, then 
\[d(z, \sigma(A))\leq Kh\left (\frac{\norm{B-A}}{K} \right )\,.\] 
Let $z\in \sigma(B)$. If $z\in \sigma(A)$, then the left-hand side of the above inequality is zero, hence
there is nothing to prove. Now assume $z\not\in\sigma(A)$. By (\ref{eq:bf60}) and the hypothesis we have
\[
\frac{1}{\norm{{B-A}}}\leq \norm{(zI-A)^{-1}}\leq \frac{1}{K}g\left(\frac{K}{d(z, \sigma(A))}\right)\,.\]
Since $g$ is strictly monotonically increasing,   
so is $\tilde{g}$. Therefore
\[\tilde{g}\left(\frac{K}{\norm{B-A}} \right )\leq\frac{K}{d(z, \sigma(A))}\,,\]
and so 
\[d(z, \sigma(A))\leq
\frac{K}{\tilde{g} \left (   \frac{K}{\norm{B-A}} \right )} 
=Kh\left (\frac{\norm{B-A}}{K} \right )\,,\]
as desired. 
\end{proof}

By combining Theorems~\ref{t1} and \ref{otbf}
we are finally able to state our spectral variation and spectral distance formulae. 
\begin{theorem}
\label{t3}
Let $w\in \mathcal{W}$. 
\begin{itemize}
\item[(i)] If $A\in E_w(H)$ is not normal, then 
\begin{equation}
 \label{et2}
 \hat d(\sigma(B), \sigma(A))\leq \nu_w(A) H_{w}\left(\frac {\norm{A-B}}{\nu_w(A)}\right)\quad (\forall B \in L(H))\,.
 \end{equation}
\item[(ii)] If $A,B\in E_w(H)$ and neither $A$ nor $B$ are normal, then 
\begin{equation}
\label{et3}
\Hdist(\sigma(A), \sigma(B)) 
    \leq m H_{w}\left(\frac{\norm{A-B}}{m}\right)\,,
\end{equation} 
where $m:=\max\{ \nu_w(A), \nu_w(B)\}$.
\end{itemize}
 Here, the function ${H}_{ w}: \R^+_0\to \R^+_0$ is defined by 
 \[H_{w}(r)=\frac{1}{{\tilde F_{\dot{\bar w}}}^{-1}(\frac{1}{r})}\,,\]
 where ${\tilde F_{\dot{\bar w}}}^{-1}$ is the inverse of ${\tilde F_{\dot{\bar w}}}:\R^+_0\to \R^+_0$
  defined by \[{\tilde F_{\dot{\bar w}}}(r)= r F_{\dot{\bar w}} (r)\,,\]
  and $F_{\dot{\bar w}}$ is the function defined in (\ref{functionc}).
\end{theorem}
\begin{proof} \noindent
\begin{itemize}
\item[(i)] By Theorem~\ref{t1}, 
\[\norm{R(A; z)}\leq \frac{1}{\nu_w(A)}\tilde F_{\dot{\bar w}}\left(\frac{\nu_w(A)}{d(z, \sigma(A))}\right)
\quad (\forall z\in \rho(A))\,,\]
so the assertion follows from the previous theorem.  
\item[(ii)] Similarly, by Theorem~\ref{t1} and Remark~\ref{rem:upper} we have 
\[\norm{R(A; z)}\leq \frac{1}{m}\tilde F_{\dot{\bar w}}\left(\frac{m}{d(z, \sigma(A))}\right)
\quad (\forall z\in \rho(A))\,,\]
and
\[\norm{R(B; z)}\leq \frac{1}{m}\tilde F_{\dot{\bar w}}\left(\frac{m}{d(z, \sigma(B))}\right)
\quad (\forall z\in \rho(B))\,,\]
so the assertion again follows by invoking the Bauer-Fike argument.   
\end{itemize}
\end{proof}

\begin{remark}\
\label{remark:sharpcompactnessbound}
\begin{itemize}
\item[(i)] Note that $\lim_{r\downarrow 0} H_{w}(r)= 0$, thus the bounds for the 
spectral variation and spectral distance
 become small when $\norm{A-B}$ is small.
 
\item[(ii)] The bounds (\ref{et2}) and (\ref{et3}) remain valid if we replace $\nu_w(A)$ and
$\nu_w(B)$ by a larger quantity, say by the upper bounds given in Proposition~\ref{lw}. 
 
 \item[(iii)] Combining (\ref{equa1}) and Theorem \ref{otbf} it follows that if 
 $A$ is a bounded normal operator on $H$, then 
\[\hat d(\sigma(B), \sigma(A)) \leq\norm{A-B}\quad (\forall B \in L(H))\,. \] 
Moreover, by symmetry it follows from the above 
that if both $A$ and $B$ are bounded normal operators then 
\[\Hdist(\sigma(A), \sigma(B)) 
    \leq\norm{A-B}\,.\]
Note that these two bounds can be thought of as limiting cases of the previous theorem, since, as is easily seen, 
we have for any $r\geq 0$ 
\[ \lim_{C\downarrow 0}CH_w \left ( \frac{r}{C} \right ) = r\,.\]     
It is in this respect that the bounds (\ref{et2}) and (\ref{et3}) are sharp. 
\end{itemize}
\end{remark}

\begin{remark}
\label{rem:SLHw}
In the case of the Schatten-Lorentz ideal and the exponential classes 
it is possible to give
rather precise estimates for the behaviour of the general bound in
Theorem~\ref{t3} for two operators which are close in operator
norm. 

We start with the Schatten-Lorentz ideal. Let 
$p\in (0,\infty)$ and let $w_k=k^{-1/p}$ for $k\in\N$. Then  
Proposition~\ref{prop:SLgrowth} yields 
\[ \log \tilde{F}_{\dot{\bar{w}}}(r) 
\sim \frac{4C^pe}{p}r^p \text{ as $r\to \infty$}\]
which, using Lemma~\ref{lem:genasym}, implies that 
\[ \tilde{F}^{-1}_{\dot{\bar{w}}}(r) 
\sim \frac{1}{C}\left ( \frac{p}{4e} \right )^{1/p}\left 
( \log r \right )^{1/p} \text{ as $r\to \infty$},\]
which, in turn, gives
\[ H_w(r)\sim  C \left ( \frac{4e}{p} \right )^{1/p} 
\abs{\log r}^{-1/p} \text{ as $r\downarrow 0$.} 
\] 

We now turn to the exponential classes. Let 
$a,\alpha\in (0,\infty)$ and let $w_k=\exp(-ak^\alpha)$ for $k\in \N$. 
Now, Proposition~\ref{prop:Egrowth} yields 
\[ \log \tilde{F}_{\dot{\bar{w}}}(r) 
\sim 4 \left ( \frac{\alpha +1}{a} \right )^{1/\alpha}\frac{\alpha}{\alpha+1}
\left ( \log r\right )^{1+1/\alpha} 
\text{ as $r\to \infty$}\]
which, using Lemma~\ref{lem:genasym}, implies that 
\[ \log \tilde{F}^{-1}_{\dot{\bar{w}}}(r) 
\sim 
4^{-\alpha/(\alpha+1)} \left (  \frac{a}{\alpha+1}\right )^{1/(\alpha
  +1)} 
\left (  \frac{\alpha +1}{\alpha}\right )^{\alpha/(\alpha +1)} 
\left ( \log r\right )^{\alpha/(\alpha+1)} 
\text{ as $r\to \infty$}\,,\]
which, in turn, gives
\[ \log H_w(r)\sim  
-4^{-\alpha/(\alpha+1)} \left (  \frac{a}{\alpha+1}\right )^{1/(\alpha
  +1)} 
\left (  \frac{\alpha +1}{\alpha}\right )^{\alpha/(\alpha +1)} 
\abs{\log r}^{\alpha/(\alpha+1)} 
\text{ as $r\downarrow 0$.} 
\] 
\end{remark}

\section{An application to inclusion regions for pseudospectra}
\label{comforth}

Pseudospectra play an important role in numerical linear algebra and perturbation theory (see, for example, \cite{Tre97, Dav07}). They are defined as follows. 

\begin{definition}
Let $A\in L(H)$ and $\epsilon>0$. The \textit{$\epsilon$-pseudospectrum of $A$}
is defined by 
\begin{equation}
\label{pseudo:def}
\sigma_{\epsilon}(A)=\sigma(A)\cup\{\,z\in\rho(A)\, :\, \norm{(zI-A)^{-1}}>1/\epsilon\,\}\,.
\end{equation}
\end{definition}

The motivation behind this definition is the observation that for any $A\in L(H)$ and any $\epsilon>0$ we have 
\begin{equation}
\label{pseudo:alt}
\sigma_\epsilon(A)=
\bigcup_{\substack{B\in L(H)\\ \norm{A-B}<\epsilon}}
\sigma(B)
\end{equation}
as is easily seen using standard perturbation theory. In other words, 
the $\epsilon$-{pseudo\-spectrum} of a bounded linear operator 
is equal to the union of the spectra of all perturbed operators with 
  perturbations that have norms strictly less than $\epsilon$. 
  
It turns out that if in the definition of the pseudospectrum (\ref{pseudo:def}) the strict inequality is replaced by a 
non-strict one, then the alternative characterisation (\ref{pseudo:alt}) holds with the strict inequality replaced by a 
non-strict one. Curiously enough, this is no longer necessarily true for operators on Banach spaces (see 
\cite{Sha09}). 

While there exist efficient methods to compute pseudospectra of
matrices (see, for example, \cite[Section~4]{Tre97}, for a brief
overview), 
the same is not true for operators on infinite-dimensional spaces, 
where the exact computation of pseudospectra can be a very challenging task. As an application of our 
resolvent bounds obtained in Section~\ref{comsecond}, we shall now provide circular inclusion regions 
for the pseudospectra of operators in a given compactness class. 

\begin{theorem}
\label{th:pseudo}
Let $\epsilon>0$. 
\begin{itemize}
\item[(i)] If $A\in L(H)$, then 
\[ \alll{z\in \mathbb{C}}{d(z, \sigma(A))<\epsilon}\subseteq
\sigma_\epsilon(A)\,.\]
\item[(ii)] If $A\in E_w(H)$ is not normal, then 
\[ \displaystyle \sigma_\epsilon(A)\subseteq \alll{z\in \mathbb{C}}{d(z,\sigma(A))<
 \nu_w(A)H_{w}\left(\frac{\epsilon}{\nu_w(A)}\right)}\,,\]
where $H_w$ is the function defined in Theorem~\ref{t3}.
\end{itemize}
\end{theorem}
\begin{proof}\ 
\begin{itemize}
\item[(i)]
The inclusion relation follows immediately from the following lower bound for the resolvent of an operator 
\[ \norm{R(A;z)}\geq \frac{1}{d(z, \sigma(A))}\,,\]
which in turn follows from 
\[  \frac{1}{d(z,\sigma(A))} =
\sup_{\lambda\in \sigma(A)}|z-\lambda|^{-1} = 
r( R(A;z) ) \leq \| R(A;z) \|\,.\]

\item[(ii)] By Theorem~\ref{t1} we have the resolvent bound 
\[\norm{R(A; z)}\leq \frac{1}{\nu_w(A)}\tilde F_{\dot{\bar w}}\left(\frac{\nu_w(A)}{d(z, \sigma(A))}\right)
\quad (\forall z\in \rho(A))\,,\]
where ${\tilde F_{\dot{\bar w}}}(r)= r F_{\dot{\bar w}} (r)\,$.

If $z\in\sigma_\epsilon(A)$, then 
\[\frac{1}{\epsilon}<\norm{R(A; z)}\leq \frac{1}{\nu_w(A)}\tilde F_{\dot{\bar w}}\left(\frac{\nu_w(A)}{d(z, \sigma(A))}\right)\,,\]
and a short calculation shows that 
\[d(z,\sigma(A))<
 \nu_w(A)H_{w}\left(\frac{\epsilon}{\nu_w(A)}\right)\,,\] as desired. 
\end{itemize}
\end{proof}

\begin{remark}\
\begin{itemize}
\item[(i)] Note that the inclusion (ii) above also follows from the characterisation (\ref{pseudo:alt}) and Theorem~\ref{t3} (i). 
\item[(ii)]
Note that the inclusion (ii) is sharp in the limiting case of normal $A$, since it reduces to
\[\sigma_\epsilon(A)=\{\,z\in \mathbb{C}\, : \,d(z, \sigma(A))<\epsilon\,\}\,.\]
\end{itemize}

\end{remark}

\section{Acknowledgements}
The research of OFB was supported by the EPSRC grant EP/R012008/1. 
Both authors would like to thank Titus Hilberdink and Eugene Shargorodsky for valuable 
feedback during the preparation of this article.


\begin{thebibliography}{OOOOO}

\bibitem[ALL01]{ALL01}
{\sc M.~Ahues, A.~Largillier, B.~V.~Limaye,}
\emph{Spectral Computations for Bounded Operators},
Roca Baton, Chapman \& Hall/CRC, 2001.


\bibitem[Ban04]{Ban04}
{\sc O.~F.~Bandtlow,}
\emph{Estimates for norms of resolvents and an application to the perturbation of spectra}, 
Math. Nachr. \textbf{267(1)} (2004), 3--11.


\bibitem[Ban08]{Ban08}
{\sc O.~F.~Bandtlow,}
\emph{Resolvent estimates for operators belonging to exponential classes}, 
Integral Equations Operator Theory, \textbf{61} (2008), 21--43.

\bibitem[BG15]{BG15}
{\sc O.~F.~Bandtlow, A.~G\"{u}ven,}
\emph{Explicit upper bounds for the spectral distance of two trace class operators},
Linear Algebra Appl. \textbf{466} (2015), 329--342.


\bibitem[BF60]{BF60}
{\sc F.~L.~Bauer, C.~T.~Fike,}
\emph{Norms and exclusion theorems}, 
Num. Math. \textbf{2} (1960), 42--53.

\bibitem[BP13]{BP13}
{\sc V.~Berinde, M.~P\u{a}curar,}
\emph{The role of the Pompeiu-Hausdorff metric in fixed point theory}, 
 Creative Mathematics and Informatics, \textbf{22(2)} (2013), 143--150.



  

\bibitem[Car21]{Car21}
{\sc T.~Carleman,}
\emph{Zur Theorie der linearen Integralgleichungen}, 
Math. Z. \textbf{9} (1921), 196--217.

\bibitem[Cha81]{Cha81}
{\sc F.~Chatelin,}
\emph{The spectral approximation of linear operators with applications to the computation of eigenelements of differential and integral operators}, 
SIAM Review, \textbf{23(4)} (1981), 495--522.

\bibitem[Cha83]{Cha83}
{\sc F.~Chatelin,}
\emph{Spectral Approximation of Linear Operators},
New York, Academic Press, 1983.


\bibitem[Dav07]{Dav07}
{\sc E.~B.~Davies,}
\emph{Linear Operators and their Spectra Vol. 106},
 Cambridge, CUP, 2007.



\bibitem[DP94]{DP94}
{\sc L.~T.~Dechevski, L.~E.~Persson,}
\emph{Sharp generalized Carleman inequalities with minimal
information about the spectrum}, 
Math. Nachr. \textbf{168} (1994), 61--77.

\bibitem[DP96]{DP96}
{\sc L.~T.~Dechevski, L.~E.~Persson,}
\emph{On sharpness, applications and generalisations of some
Carleman type inequalities}, 
T\^{o}hoku Math. J. \textbf{48} (1996), 1--22.

\bibitem[Dos01]{Dos01}
{\sc M.~R.~Dostani\'{c}, }
\emph{Estimates for the norm of powers of Volterra's operator through its singular values}, 
Math.\ Z. \textbf{236(3)} (2001), 453--459.

\bibitem[DS63]{DS63}
{\sc N. ~Dunford, J.~T.~Schwartz, }
\emph{Linear Operators Vol. 2}, 
New York, Interscience, 1963.

  




\bibitem[Gil95]{Gil95}
{\sc M.~I.~Gil\cp,}
\emph{Norm Estimations for Operator-Valued Functions and Applications}, 
New York, Marcel Dekker, 1995.

\bibitem[Gil03]{Gil03}
{\sc M.~I.~Gil\cp,}
\emph{Operator Functions and Localization of Spectra}, 
Berlin, Springer, 2003.

\bibitem[Gil12]{Gil12}
{\sc M.~I.~Gil\cp,}
\emph{Norm estimates for resolvents of non-selfadjoint operators having Hilbert-Schmidt inverse ones}, 
Math. Commun. \textbf{17} (2012), 599--611. 

\bibitem[Gil14]{Gil14}
{\sc M.~I.~Gil\cp,}
\emph{Resolvents of operators inverse to Schatten-von Neumann ones}, 
Ann. Univ. Ferrara  Sez. VII Sci. Mat. \textbf{60(2)} (2014), 363--376. 

\bibitem[GGK90]{GGK90}
{\sc I.~Gohberg, S.~Goldberg, M.~A.~Kaashoek,}
\emph{Classes of Linear Operators Vol. 1}, 
Basel, Birkh\"auser Verlag, 1990.



\bibitem[GK69]{GK69}
{\sc I.~Gohberg, M.~G.~Krein,}
\emph{Introduction to the Theory of Linear Non-Selfadjoint Operators}, 
Providence, AMS, 1969.

\bibitem[Han10]{Hansen}
{\sc A.~C.~Hansen,}
\emph{Infinite-dimensional numerical linear algebra: theory and applications},
Proc. R. Soc. Lond. Ser. A Math. Phys. Eng. Sci. \textbf{466}
(2010), 3539--3559.


\bibitem[Hen62]{Hen62}
{\sc P.~Henrici,}
\emph{Bounds for iterates, inverses, spectral variation and fields of values of non-normal matrices}, 
Num. Math. \textbf{4} (1962), 24--40.

\bibitem[HP11]{HP11}
{\sc D.~Hinrichsen, A.~J.~Pritchard,}
\emph{Mathematical Systems Theory I: Modelling, State Space Analysis, Stability and Robustness Vol. 48}, 
 Heidelberg, Springer, 2011.




\bibitem[Kat76]{Kat76}
{\sc T.~Kato,}
\emph{Perturbation Theory for Linear Operators}, 
Berlin, Springer, 1976.







\bibitem[Nel82]{Nel82}
{\sc E.~Nelimarkka,}
\emph{On $\lambda(P,N)$-nuclearity and operator ideals}, 
Math. Nachr. \textbf{99} (1982), 231--237.


 
\bibitem[New51]{New51}
{\sc J.~D.~Newburgh,}
\emph{The variation of spectra}, Duke Math. J. \textbf{18} (1951), 165--176.


 




 
\bibitem[Pel85]{Pel85} 
{\sc V.~V.~Peller,} 
\emph{A description of Hankel operators of the class $\mathfrak{S}_p$ for
$p>0$, an investigation of the rate of rational approximation and other 
applications}, 
Math. USSR Sbornik \textbf{50(2)} (1985), 465--494.

\bibitem[Pie80]{Pie80}
{\sc A.~Pietsch,}
\emph{Operator Ideals}, 
Amsterdam-New York, North-Holland, 1980.

\bibitem[Pie86]{Pie86}
{\sc A.~Pietsch,}
\emph{Eigenvalues and s-Numbers}, 
Cambridge, CUP, 1986.


\bibitem[Pok85]{Pok85}
{\sc A.~Pokrzywa,}
\emph{On continuity of spectra in norm ideals}, Linear Algebra Appl. \textbf{69} (1985), 121--130.


\bibitem[Sha09]{Sha09}
{\sc E.~Shargorodsky,}
\emph{On the definition of pseudospectra}, Bull. London Math. Soc. \textbf{41} (2009), 524--534.
 
\bibitem[Sim77]{Sim77}
{\sc B.~Simon,}
\emph{Notes on infinite determinants of Hilbert space operators}, Adv. Math. \textbf{24} (1977), 244--273.


\bibitem[Tre97]{Tre97}
{\sc L.~N.~Trefethen,}
\emph{Pseudospectra of linear operators}, 
 SIAM Review, \textbf{39(3)} (1997), 383--406.
 



 


\end{thebibliography}
\end{document}